\documentclass[12pt]{amsart}
\usepackage[utf8]{inputenc}
\usepackage[margin=1in]{geometry}

\RequirePackage{color}
\usepackage{xcolor}

\makeatletter

\usepackage{amsmath,amsfonts,amssymb,amsthm,epsfig,epstopdf,titling,url,array,mathtools,faktor,paralist, xfrac, mathrsfs, enumitem}
\usepackage[new]{old-arrows}
\usepackage[foot]{amsaddr}

\theoremstyle{plain}
\newtheorem{thm}{Theorem}[subsection]
\newtheorem*{thmm}{Theorem}
\newtheorem{thmx}{Theorem}

\newtheorem{lem}[thm]{Lemma}
\newtheorem{prop}[thm]{Proposition}
\newtheorem{cor}[thm]{Corollary}

\theoremstyle{definition}
\newtheorem{defn}{Definition}[subsection]

\newtheorem{rmk}[thm]{Remark}

\DeclareMathOperator{\SI}{SI}
\DeclareMathOperator{\GL}{GL}
\DeclareMathOperator{\SL}{SL}
\DeclareMathOperator{\rep}{{\rm rep}}
\DeclareMathOperator{\K}{\mathbb{K}}
\DeclareMathOperator{\img}{Im}
\DeclareMathOperator{\Hom}{Hom}
\DeclareMathOperator{\Ext}{Ext}
\DeclareMathOperator{\remainder}{rem}
\newcommand{\rem}[1]{{\remainder(#1)}}

\begingroup\lccode`~=`& \lowercase{\endgroup
\newcommand{\spmat}[1]{%
  \left[
  \let~=&
  \begin{matrix}#1\end{matrix}
  \right]
}}

\definecolor{bettergreen}{RGB}{0,150,0}
\definecolor{alizarin}{rgb}{0.85, 0.15, 0.26}
\definecolor{azure}{rgb}{0.0, 0.5, 1.0}
\definecolor{upmaroon}{rgb}{0.80, 0.07, 0.07}

\usepackage{tikz}
\usepackage{tikz-cd}
\usetikzlibrary{arrows,positioning,scopes,trees}

 \usepackage{hyperref}
 \hypersetup
 {
     colorlinks=true,
     linktocpage=true,
     linkcolor=upmaroon,
     filecolor=magenta,
     citecolor=bettergreen,
     urlcolor=azure,
     pdftitle={Semi-Invariant Rings: UFD and Codimension one Orbits}
 }

\makeatletter
\newcommand{\doublewidetilde}[1]{{%
  \mathpalette\double@widetilde{#1}%
}}
\newcommand{\double@widetilde}[2]{%
  \sbox\z@{$\m@th#1\widetilde{#2}$}%
  \ht\z@=.9\ht\z@
  \widetilde{\box\z@}%
}
\makeatother

\usepackage[style=alphabetic]{biblatex}
\newcommand{\C}{{\mathcal C}}

\title{Semi-Invariant Rings: UFD and Codimension one Orbits}
\author{Charles Paquette $^{1,2}$, Deepanshu Prasad $^{1}$ AND David Wehlau $^{1,2}$}
\address{$^1$Department of Mathematics and Statistics, Queen's University, Kingston ON, Canada}
\address{$^2$Department of Mathematics and Computer Science, Royal Military College of Canada, Kingston ON, Canada}
\email{charles.paquette.math@gmail.com, 19dp7@queensu.ca, david.wehlau@queesnu.ca}

\begin{document}

\begin{abstract}
    Let $A$ be a finite dimensional associative $\mathbb{K}$-algebra over an algebraically closed field $\mathbb{K}$ of characteristic zero. To $A$, we can associate its basic form that is given by a quiver $Q = (Q_0, Q_1)$ with an admissible ideal $R$. For a dimension vector $\beta$, we consider an irreducible component $\mathcal{C}$ of the module variety of $\beta$-dimensional representations of $A$. The reductive group ${\rm GL}_\beta(\mathbb{K}):= \prod_{i \in Q_0}{\rm GL}_{\beta_i}(\mathbb{K})$ acts on $\mathcal{C}$ by change of basis, and has a unique closed orbit. We consider the corresponding ring of semi-invariants $\SI(Q, \mathcal{C})$. We prove that if $\mathcal{C}$ is factorial and has maximal orbits of codimension one, then $\SI(Q, \mathcal{C})$ is a complete intersection and is not multiplicity free.  If $\mathcal{C}$ is not factorial, then this conclusion does not necessarily hold. We present examples showing that the codimension of the complete intersection can be arbitrarily large. Finally, we interpret our results in the case of hereditary algebras.
\end{abstract}

\maketitle

\section{Introduction}
We let $A$ denote a finite dimensional (associative and unital) algebra over an algebraically closed field $\mathbb{K}$ of characteristic zero. 
It is well known that up to Morita equivalence, we may assume that $A$ is an admissible quotient of a path algebra of a finite quiver. Thus, there exists a quiver $Q = (Q_0, Q_1)$, which is uniquely determined by $A$, and an admissible ideal $R$ of $\mathbb{K}Q$ such that $A\cong\mathbb{K}Q/R$. We will identify left $A$-modules with representations of $Q$ satisfying the relations from $R$.

Let $\beta$ be a dimension vector, that is, an element $(\beta_i)_{i \in Q_0}$ in the positive part $(\mathbb{Z}_{\ge 0})^{Q_0}$ of the Grothendieck group $K_0(A)$ of $A$.
We consider the representation space $\rep_{\beta}(Q, R)$, which parametrizes the $\beta$-dimensional representations of $(Q,R)$.  This is an affine variety, which is not necessarily irreducible. The general linear group $\GL_{\beta}(\K)\coloneqq\prod_{i \in Q_0}\GL_{\beta_i}(\K)$ acts by change of bases, so that the orbits corresponds to the isomorphism classes. This provides a geometric approach to study the representations of $(Q,R)$.

In \parencite{King}, King introduced moduli spaces of representations of quivers (with relations). In his work, King developed an algebraic notion of stability, which allows one to interpret Mumford's stability in the setting of quiver representations. The multiplicative characters of $\GL_{\beta}(\K)$ are in bijection with  $\Hom_{\mathbb{Z}}(\mathbb{Z}^{Q_0},\mathbb{Z})$, the space of weights. For a stability parameter (or weight) $\theta \in \Hom_{\mathbb{Z}}(\mathbb{Z}^{Q_0},\mathbb{Z})$, one considers the $\theta$-semistable representations in some irreducible component $\mathcal{C}$ of $\rep_{\beta}(Q,R)$. Up to an equivalence relation (called $S$-equivalence), these representations form a projective variety  $\mathcal{M}(\mathcal{C})^{\theta-ss}$, the moduli space of $\theta$-semistable representations in $\mathcal{C}$. These moduli spaces are projective varieties corresponding to certain subrings of the so-called rings of semi-invariants of $\mathcal{C}$.  The study of the semi-invariant ring $\SI(Q,\mathcal{C})$ provides a way to understand these moduli spaces. 
Specifically King showed that
\begin{equation}
\label{modsp}
    \mathcal{M}(\mathcal{C})^{\theta-ss}\coloneqq\mathrm{Proj}\left(\bigoplus_{n\geq0}\SI(Q,\mathcal{C})_{n\theta}\right)
\end{equation}
where $\bigoplus_{n\geq0}\SI(Q,\mathcal{C})_{n\theta}$ is a direct sum of weight spaces forming a graded subring of $\SI(Q,\mathcal{C})$, is a certain GIT-quotient.

A famous result of Sato-Kimura states that:
\begin{thmm}[Sato-Kimura]
    Let $G$ be a connected linear algebraic group and $V$ be a rational representation of $G$ such that the action has a dense open orbit. Then, the semi-invariant ring is a polynomial ring. (see \parencite[Lemma 9.4.2]{weyderk} or \parencite[Proposition 5, page-60]{satkim}))
\end{thmm}
This theorem applies when $R=0$, since $\rep_{\beta}(Q)$ is an affine space. Therefore, in this case when $\rep_{\beta}(Q)$ contains an open $\GL_{\beta}$-orbit, the ring of semi-invariants of the $\beta$-dimensional representations is a polynomial ring.  In general, if $\mathcal{C}$ contains an open $\GL_{\beta}$-orbit, then $\SI(Q,\mathcal{C})$ is not necessarily a polynomial ring. However, one can adapt the proof of Sato and Kimura (see also \parencite{ChinKinWey}) to prove that if $\C$ is an orbit closure, then the ring of semi-invariants is multiplicity-free, that is, each weight space of $\SI(Q,\mathcal{C})$ is at most one dimensional.

A next natural step is to consider the case where $\mathcal{C}$ contains a codimension one orbit. In this case, if $\mathcal{C}$ is not an orbit closure, then $\mathcal{C}$ contains an open (dense) set of codimension one orbits. This leads us to study the semi-invariant ring $\SI(Q,\mathcal{C})$ when $\mathcal{C}$ contains an open dense set of codimension one $\GL_{\beta}$-orbits.  We require the condition that $\mathbb{K}[\mathcal{C}]$ is a UFD. 
Under this assumption, our methods allow us to describe the generators of $\SI(Q,\mathcal{C})$. This leads us to one of our main results:
\begin{thmx}
\phantomsection\label{ThmA}
    Let $\mathcal{C}$ be an irreducible component of $\rep_{\beta}(Q,R)$ which is not an orbit closure but which contains a codimension one orbit. Suppose further that the coordinate ring $\mathbb{K}[\mathcal{C}]$ is a UFD. Then, the semi-invariant ring $\SI(Q,\mathcal{C})$ is a complete intersection whose codimension is equal to  $\#\lbrace\text{monomials of weight $\chi$}\rbrace-2$, where $\chi$ is the unique minimal weight such that $\mathrm{dim}_{\mathbb{K}}\,\SI(Q,\mathcal{C})_{\chi}=2$.
\end{thmx}

The above theorem is stated in the context of representations of quivers with relations.  Our techniques can be applied in the following more general setting.  Let $G$ be a connected reductive algebraic group acting regularly on a factorial variety $\mathcal{C}$.  We can define the ring of semi-invariants as $\mathbb{K}[\mathcal{C}]^{[G,G]}$ where $[G,G]$ is the derived subgroup of $G$.  As in the quiver case this ring has a weight space decomposition given by the 
multiplicative characters of $G$.  Theorem~\ref{ThmA} applies when $\mathcal{C}$ has a single closed orbit and its maximal orbits are of co-dimension 1.  

\medskip
For hereditary algebras, the variety $\rep_\beta(Q)$ is irreducible and its structure depends on the so-called canonical decomposition of $\beta$, which is a decomposition of $\beta$ into a sum of special dimension vectors known as Schur roots. These roots fall into three types, the real, isotropic and imaginary Schur roots. A description of the semi-invariant ring is known when the decomposition involves only real Schur roots or a single isotropic Schur root. 

If the canonical decomposition of the dimension vector $\beta$ includes real Schur roots and exactly one isotropic Schur root, then we call such dimension vector $\emph{almost prehomogenous}$. We study some  properties of these dimension vectors in Proposition \ref{alprehom}. By applying Theorem~\ref{ThmA} to a quiver without relations, we obtain a description of the semi-invariant ring $\SI(Q, \beta)$, where $\beta$ is almost prehomogeneous. This yields the second main result of this article and corroborates, albeit in a slightly weakened form, the result outlined in \parencite[Corollary 6.3]{weypaq}.
\begin{thmx}
    Let $Q$ be an acyclic quiver without any relations and let $\beta$ be a dimension vector.
    Suppose that the canonical decomposition of $\beta$ contains exactly one isotropic Schur root and that all the other roots occurring in the decomposition are real Schur roots. Then the semi-invariant ring $\SI(Q,\beta)$ is a complete intersection.
\end{thmx}

Finally, we consider a more general case where the orbits of maximal dimension are not necessarily of codimension one:
\begin{thmx}
    If there is an irreducible component $\mathcal{C}\subseteq \rep_{\beta}(Q,R)$ such that $\mathcal{C}$ is not an orbit closure and the coordinate ring $\mathbb{K}[\mathcal{C}]$ is a UFD, then the semi-invariant ring $\SI(Q,\mathcal{C})$ is not multiplicity free.
\end{thmx}
 A nice consequence of $\SI(Q,\mathcal{C})$ failing to be multiplicity free is that there exists a dimension vector $\beta'$ such that $\rep_{\beta'}(Q,R)$ contains infinitely many non-isomorphic bricks $\lbrace B_{\ell}\rbrace_{\ell\in\Lambda}$, i.e., infinitely many non-isomorphic modules $B_\ell$ with $\mathrm{End}_{A}(B_{\ell})\cong\mathbb{K}$.

\section{Preliminaries}

In this section, we recall some key notions on quiver representations and their invariant theory. For more details, the reader is referred to \cite{weyderk}.

A quiver $Q=(Q_0,Q_1)$ is a directed graph with $Q_0$ being the set of vertices, and $Q_1$ being the set of arrows. Such a quiver comes equipped with two
maps $t,h : Q_ 1 \rightarrow Q_ 0$ with $ta \coloneqq t(a)$, the tail of the arrow $a$, and $ha \coloneqq h(a)$, the head of the arrow.

A path $p$ in $Q$ is a sequence of arrows $p=a_r \cdots a_1$ such that $ha_i = ta_{i+1}$ for $1 \le i \le r-1$. The path $p$ is said to have length $r$.  We can extend the functions $h, t$ to paths by defining $hp = h a_r$ and $tp = ta_1$. For each $x \in Q_0$, we also define a trivial path $e_x$ of length $0$ with $he_x = te_x = x$. 

The path algebra $\mathbb{K}Q$ of the quiver $Q$ is constructed as follows. As a $\K$-vector space, it has a basis given by all paths in $Q$ (including the trivial paths). The multiplication of two paths $p, q$ is given by their concatenation, written as  $qp$, if $h(p) = t(q)$, and is zero otherwise.  

Recall that an ideal $R$ of $\mathbb{K}Q$ is \emph{admissible} if there exists a positive integer $t$ such that $R$ contains all paths of length $t$ and $R$ can be generated by elements that can be written as linear combination of paths of length at least two. The importance in admissible ideals lies in the well known fact that any basic finite dimensional algebra over $\K$ can be realized as an admissible quotient of a path algebra of a finite quiver over $\K$.

 An element $u$ in $\mathbb{K}Q$ is said to be \emph{uniform} if it is a linear combination of parallel paths, that is, paths sharing the same head and the same tail. We note that an admissible ideal can always be generated by uniform elements. Recall that a representation $V$ of $Q$ over $\K$ is given by $V = (\{V_x\}_{x \in Q_0}, \{V(a)\}_{a \in Q_1})$ where the $V_x$ are finite dimensional $\K$-vector spaces and $V(a): V_{ta} \to V_{ha}$ is a $\K$-linear map for each arrow $a \in Q_1$. For a path $p = a_r \cdots a_1$ in $Q$ and a representation $V$, we define $V(p)$ to be the composition $V(p) = V(a_r) \cdots V(a_1)$, which is a map from $V_{tp}$ to $V_{hp}$.  If $u = \sum_{i=1}^t\lambda_i p_i$ is a uniform element, then we define $V(u) = \sum_{i=1}^t\lambda_iV(p_i)$. A representation $V$ is said to be annihilated by $R$ if $V(u)=0$ for every uniform element $u$ in $R$.

In this paper, we assume that our algebra $A$ is given by $A = \mathbb{K}Q/R$ where $Q$ is a finite quiver and $R$ is an admissible ideal of $\mathbb{K}Q$.
We denote by ${\rm mod}\, A$ the category of finite dimensional left $A$-modules. It is well known that the category ${\rm mod}\, A$ is equivalent to the category ${\rep}(Q,R)$ of finite dimensional $\K$-representations of $Q$ annihilated by $R$. 

The dimension vector $\beta \in\mathbb{N}^{Q_0}$ of a representation $V \in \rep(Q,R)$ is defined by $\beta(x) = \dim_{\mathbb {K}} V_x$. For each $x \in Q_0$ we fix a basis for the $\beta(x)$ dimensional vector space $V_x$, 
and define
\[
    \rep_{\beta}(Q, R)\coloneqq\left\{(V(a))_{a\in Q_1} \in\prod_{a\in Q_1}\Hom_{\mathbb{K}}(V_{ta},V_{ha})\;\bigg|\; V(u) = 0,\text{ for every uniform } u \in R\right\}.
\]

We have a regular action of the algebraic group 
\[
    \GL_{\beta}\coloneqq\prod_{x\in Q_0}\GL_{\beta(x)}
\]
on $\rep_{\beta}(Q, R)$ given by $g\cdot V = (g_{ha}V(a)g_{ta}^{-1})_{a\in Q_1}$ where $(V(a))_{a\in Q_1} \in \rep_{\beta}(Q, R)$ and $(g_x)_{x \in Q_0} \in \GL_{\beta}$.
The orbits correspond to isomorphism classes of $\beta$-dimensional representations of $\mathbb{K}Q/R$ modules.

The representation space $\rep_{\beta}(Q,R)$ is an affine variety which is not necessarily irreducible. Consider one of its irreducible components $\mathcal{C}$. The action of $\GL_{\beta}$ on $\rep_{\beta}(Q,R)$ induces an action on the elements of the coordinate ring $\mathbb{K}[\mathcal{C}]$ as follows: for every $g\in\GL_{\beta},f\in\mathbb{K}[\mathcal{C}]$ and $x\in\mathcal{C}$, we have $g\cdot f(x)\coloneqq f(g^{-1}x)$.
Then,
\[
    \SI(Q,\mathcal{C})\coloneqq \mathbb{K}[\mathcal{C}]^{\SL_{\beta}},\text{ where }\SL_{\beta}\coloneqq\prod_{x\in Q_0}\SL_{\beta(x)}
\]
is called the ring of semi-invariant functions on $\mathcal{C}$.

For every weight $\sigma\in\mathbb{Z}^{Q_0}$, we can define a multiplicative character of $\GL_\beta$ as follows:
\[
    \begin{array}{ccccc}
         \chi_\sigma & : & \GL_\beta & \longrightarrow & \mathbb{K}^*\\
         \, & \, & (g_x)_{x\in Q_0} & \longmapsto & \prod_{x\in Q_0} \det(g_x)^{\sigma(x)}
    \end{array}
\]
The semi-invariant ring $\SI(Q,\mathcal{C})$ has a decomposition into weight spaces,
\[
    \SI(Q,\mathcal{C})=\bigoplus_{\sigma}\SI(Q,\mathcal{C})_{\sigma}
\]
where
\[
    \SI(Q,\mathcal{C})_{\sigma}\coloneqq\lbrace f\in\mathbb{K}[\mathcal{C}]\mid g\cdot f=\chi_{\sigma}(g)\,f,\text{ for every }g\in\GL_\beta\rbrace
\]
is called the space of semi-invariants of weight $\sigma$. A semi-invariant in a given weight space is said to be \emph{homogeneous}.

\section{Codimension one orbits and UFD}

Let $\beta$ be a dimension vector, and let $\mathcal{C}\subseteq {\rm rep}_{\beta}(Q,R)$ be an irreducible component.
 \begin{lem}
 \phantomsection\label{units}
     The units of the coordinate ring $\mathbb{K}[\mathcal{C}]$ are the elements of $\mathbb{K}^*$.
 \end{lem}
 \begin{proof}
     Clearly, any element of $\mathbb{K}^*$ is a unit in $\mathbb{K}[\mathcal{C}]$.\par
     Let $u\in\mathbb{K}[\mathcal{C}]$ be a unit. Therefore, there exists a unit $v\in\mathbb{K}[\mathcal{C}]$ such that $uv=1$, i.e., $u(x)v(x)=1$, for every $x\in\mathcal{C}$. By \parencite[Lemma 5.4, Chap.~6]{SanRitt}, we know that $u\in\SI(Q,\mathcal{C})_\chi$ for some character $\chi$.
    
      Thus $u(g^{-1}x)=\chi(g)u(x)$, for every $x\in\mathcal{C}$ and every $g\in\GL_{\beta}$. Multiplying by $v(x)$, we get,
     \[
         u(g^{-1}x)v(x)=\chi(g)u(x)v(x)=\chi(g)
     \]
     Take $x=0$, i.e., the point corresponding to the semi-simple representation. Therefore, $\chi(g)=u(0)v(0)=1$ for all $g \in \GL_\beta$ and $\chi$ is the trivial character. In particular, $u$ is an invariant function. Hence, $u\in\mathbb{K}^*$.
 \end{proof}
The next proposition gives a characterisation of the irreducible elements of $\SI(Q,\mathcal{C})$ when the coordinate ring $\mathbb{K}[\mathcal{C}]$ is a UFD. Recall an affine subvariety $W$ of $\mathcal{C}$ is \emph{pure} if all of its irreducible components have the same dimension.
 \begin{prop}
 \phantomsection\label{genSI}
     Assume that  $\mathbb{K}[\mathcal{C}]$ is a UFD and let $W\subseteq\mathcal{C}$ be a affine subvariety.  Then $W$ is a pure codimension one $\GL_{\beta}$-stable subvariety if and only if $W$ is the zero set of a homogeneous semi-invariant, i.e., $W=\mathcal{V}(f)$, for some homogeneous $f\in\SI(Q,\mathcal{C})$.
 \end{prop}
 \begin{proof}
    The sufficiency is trivial.  For the other direction, since $\GL_\beta$ is connected, it is sufficient to consider the case where $W$ is irreducible. Since $\mathbb{K}[\mathcal{C}]$ is a UFD and $W$ is of codimension one and irreducible, we conclude that $W=\mathcal{V}(f)$, for some irreducible element $f\in\mathbb{K}[\mathcal{C}]$. Since $W$ is $\GL_{\beta}$-stable, we have $\mathcal{V}(g\cdot f)=\mathcal{V}(f)$, for every $g\in\GL_{\beta}$. By Hilbert's Nullstellensatz, we get, $\sqrt{(g\cdot f)}=\sqrt{(f)}$. Using the fact that $\mathbb{K}[\mathcal{C}]$ is a UFD, and that $g\cdot f$ and $f$ are irreducible, and hence prime, we see that $g\cdot f$ and $f$ are associates, i.e., $g\cdot f=u_{g}f$, for some unit $u_{g}\in\mathbb{K}[\mathcal{C}]$.
   
    By Lemma \ref{units}, we know that $u_{g}\in\mathbb{K}^*$. It follows that $\chi:\GL_{\beta}\rightarrow \mathbb{K}^{*}$, defined by $\chi(g)\coloneqq u_{g}$ is a multiplicative character of $\GL_{\beta}$, and hence, $f$ is a homogeneous semi-invariant. 
 \end{proof}

Let $\mathcal{C}$ be an irreducible component of $\rep_{\beta}(Q,R)$. For the rest of this section we assume that $\mathbb{K}[\mathcal{C}]$ is a UFD and that $\mathcal{C}$ is not an orbit closure and $\mathcal{C}$ contains a codimension one orbit. 

By upper semicontinuity of orbit dimension, it follows that the union $U$ of all codimension one orbits forms an open subset of $\mathcal{C}$. Let $\mathcal{I}$ be a complete set of representatives of these orbits. Then $U = \bigcup_{M \in \mathcal{I}}\mathcal{O}(M)$. 

Our goal is to describe the semi-invariant ring $\SI(Q,\mathcal{C})$ in the above setting. By Hilbert's finiteness theorem the ring $\SI(Q,\mathcal{C})$ is a finitely generated $\mathbb{K}$-algebra (see \parencite[Corollary 9.2.8]{weyderk}). Therefore, we first determine a set of generators, and then the algebraic relations among these generators.
\subsection{Generators}
In this subsection, we describe the generators of the semi-invariant ring $\SI(Q,\mathcal{C})$.
Decompose the complement $U^\mathsf{c}$ of $U$ into irreducible components, namely 
\[
    U^\mathsf{c}=\mathcal{Z}_1\cup\dotsb\cup\mathcal{Z}_r\cup\mathcal{Z}_{r+1}\cup\dotsb\cup\mathcal{Z}_t
\]
where we assume that $\mathrm{codim}\, \mathcal{Z}_i=1$, for $1\leq i\leq r$, and $\mathrm{codim}\,\mathcal{Z}_i>1$, for $i>r$. We note that $r$ may be zero. With the above notation we are ready to prove the following key lemma.
\begin{lem}
\phantomsection\label{2dimweight}
    There exists at least one weight $\nu$ such that $\dim_{\mathbb{K}}\SI(Q,\mathcal{C})_{\nu}\geq2$.
\end{lem}
\begin{proof}
    Suppose, by way of contradiction, that there does not exist any weight $\sigma$ such that $\dim_{\mathbb{K}}\SI(Q,\mathcal{C})_{\sigma}\geq2$. Since $\mathbb{K}[\mathcal{C}]$ is a UFD and the orbits $\mathcal{O}(M)$, for $M\in\mathcal{I}$, are of codimension one, therefore, for every $M\in\mathcal{I}$, we can write $\overline{\mathcal{O}(M)}=\mathcal{V}(f_{M})$, for some $f_{M}\in\mathbb{K}[\mathcal{C}]$. By Proposition \ref{genSI}, we know that $f_{M}$ is an irreducible semi-invariant, for every $M\in\mathcal{I}$.
    By our assumption, the weights of the $f_{M}$ are all distinct. 
    By Hilbert's finiteness theorem, there exists a finite set of homogeneous semi-invariants $x_1,x_2,\dotsc,x_n$ which generate $\SI(Q,\mathcal{C})$.
    Clearly we may assume this set of generators is a linearly independent set.
    Now, there exists some $M'\in\mathcal{I}$ such that $f_{M'}$ is not a scalar multiple of $x_l$, for any $l\in\lbrace1,2,\dotsc,n\rbrace$. Since $f_{M'}$ is irreducible, it cannot be a monomial in the $x_i$. Therefore, $f_{M'}$ is a linear combination of at least two distinct monomials in the $x_i$. If $\nu$ denotes the weight of these two semi-invariants, we see that $\dim_{\mathbb{K}}\,\SI(Q,\mathcal{C})_{\nu}\geq2$.
\end{proof}

For $1\leq i\leq r$, $\mathcal{Z}_i$ is an irreducible closed set of codimension one in $\mathcal{C}$. Since $\mathbb{K}[\mathcal{C}]$ is a UFD, we know there exists $f_i\in\mathbb{K}[\mathcal{C}]$ such that $\mathcal{Z}_i=\mathcal{V}(f_i)$, for $1\leq i\leq r$. 
\begin{lem}
    The irreducible elements $f_1,f_2,\dotsc,f_r$ are homogeneous elements of $\SI(Q,\mathcal{C})$.
\end{lem}
\begin{proof}
    Since $\GL_\beta$ is connected (hence irreducible), each orbit (of $\mathcal{Z}_1 \cup \cdots \cup \mathcal{Z}_r$) is irreducible. Hence, we see that the elements of $\GL_\beta$ fix each irreducible component $\mathcal{Z}_i$. Now, we apply Proposition \ref{genSI}.
\end{proof}
We define the cone of weights
\[
    \Sigma(Q,\mathcal{C})\coloneqq\lbrace\sigma\in\mathbb{Z}^{Q_0}\mid\SI(Q,\mathcal{C})_{\sigma}\not=0\rbrace
\]
We have a partial ordering on the weights $\sigma\in\Sigma(Q,\mathcal{C})$. For $\sigma_1,\sigma_2\in\Sigma(Q,\mathcal{C})$, we say $\sigma_1\succeq\sigma_2$, if there exists a semi-invariant in $\SI(Q,\mathcal{C})_{\sigma_1}$ divisible by some semi-invariant in $\SI(Q,\mathcal{C})_{\sigma_2}$. Now, it follows from Lemma \ref{2dimweight} that there exists a weight $\chi$ such that $\SI(Q, \mathcal{C})_\chi$ is at least two dimensional and $\chi$ is minimal with respect to this property. Let us fix such a weight $\chi$.
\begin{lem} 
\phantomsection\label{LemmaReducingU}
    By possibly reducing $U$ to a smaller open non-empty set, we may assume that $U^\mathsf{c}$ has $r$ irreducible components of codimension one: 
    $\mathcal{V}(f_1),\mathcal{V}(f_2),\dotsc,\mathcal{V}(f_r)$ where $r\geq2$ and such that there are two linearly independent monomials in $f_1, f_2,\dots,f_r$ having weight $\chi$.
\end{lem}
\begin{proof}
    Starting with the given set $U$, suppose that no two linearly independent
    monomials in the $f_i$ share same weight.  
   Take two linearly independent semi-invariants of weight $\chi$, say $y$ and $z$. Notice that the $\GL_\beta$-stable irreducible codimension one closed subsets in $\mathcal{C}$ are precisely the $\mathcal{V}(f_i)$ and the orbit closures of representations in $\mathcal{I}$. We consider the vanishing sets of $y$ and $z$ and decompose them into their irreducible components as follows:
    \begin{align*}
        \mathcal{V}(y)=\cup_{i\in S}\mathcal{V}(f_i)\,\bigcup\,\cup_{M\in \mathcal{I}_1}\overline{\mathcal{O}(M)}\\
        \mathcal{V}(z)=\cup_{i\in T}\mathcal{V}(f_i)\,\bigcup\,\cup_{M\in \mathcal{I}_2}\overline{\mathcal{O}(M)}
    \end{align*}
    where $S,T\subseteq\lbrace1,2,\dotsc,r\rbrace$ and $\mathcal{I}_1,\mathcal{I}_2\subseteq\mathcal{I}$ are finite. By our assumption, we may assume that $y$ is not a scalar multiple of a monomial in the $f_i$, hence, $\mathcal{I}_1$ is non-empty. 
    
    Now, we shrink the open set $U$ by removing the orbit closures $\overline{\mathcal{O}(M)}$, for $M\in\mathcal{I}_1\cup\mathcal{I}_2$, from $U$. 
    Therefore, we get a new open set in $\mathcal{C}$ which is a union of orbits of codimension one, say $U'=\bigcup_{M\in\mathcal{I}'}\mathcal{O}(M)$, and a new set of semi-invariants $f_j$ with the desired property that there are at least two monomials in the $f_j$ which are semi-invariants of the same weight. In particular, this shrinking procedure allows us to assume that $r\geq2$.
\end{proof}

From now on, we assume that the open set $U$ satisfies Lemma \ref{LemmaReducingU} and that $\mathcal{I}$ denotes a complete set of representatives of the orbits in $U$. Furthermore, the $\mathcal{V}(f_i)$ for $1 \le i \le r$ are the codimension one irreducible components of $U^\mathsf{c}$. Observe that if $\chi \succ \sigma$, then the weight space 
$\SI(Q,\mathcal{C})_\sigma$ is spanned by a single monomial in $f_i$.
\begin{defn}
    Two linearly independent monomials $p,q$ in the $f_i$, of equal weight  
    form a $\emph{monomial pair}$, denoted by $(p,q)$.
\end{defn}
\begin{lem}
\phantomsection\label{p,q}
    Let $(p,q)$ be a monomial pair. Then for every $M \in \mathcal{I}$, we have that $p(M)\ne 0$ and $q(M) \ne 0$.
\end{lem}
\begin{proof}
    Recall that $p$ and $q$ are monomials in the $f_i$. Since the $f_i$ are semi-invariants and $\mathcal{O}(M)\subseteq\overline{\mathcal{O}(M)}$ is dense and irreducible, if $p$ vanishes on some $\mathcal{O}(M)$, and hence on $\overline{\mathcal{O}(M)}$, then there exists $j$ such that $f_j$ divides $p$ and $f_j(\overline{\mathcal{O}(M)})=0$. This cannot happen by the choice of the $f_i$. Therefore, the statement follows.
\end{proof}
We fix a monomial pair $(p,q)$, and for $\alpha\in\mathbb{K}$, we define $h_\alpha\coloneqq\alpha p-q$. Then, $h_\alpha$ is a non-constant semi-invariant of weight $\chi$.
\begin{rmk}
    By Lemma \ref{p,q}, we know $p$ and $q$ do not vanish anywhere on $U$. 
    Therefore, for every $M\in\mathcal{I}$, we can find $\alpha$ such that $h_\alpha$ vanishes on $\overline{\mathcal{O}(M)}$. In particular, if $0=h_\alpha(M)=\alpha p(M)-q(M)$, then we get,
    \[
        \alpha=\frac{q(M)}{p(M)}.
    \]
\end{rmk}
By Lemma \ref{p,q}, we know that $\alpha\not=0$. Therefore, we get a map
\[
    \begin{array}{cccc}
    \phi: & \mathcal{I} & \longrightarrow & \mathbb{K}^*\\
    \textbf{ } & M & \longmapsto & \alpha=\frac{q(M)}{p(M)}
    \end{array}
\]
\begin{prop}
\phantomsection\label{irred}
    If $\alpha\in\img(\phi)$, then $h_\alpha$ is irreducible and there exists $M \in \mathcal{I}$ such that $\mathcal{V}(h_\alpha) = \overline{\mathcal{O}(M)}$.
\end{prop}
\begin{proof}
    Note that $h_\alpha\in\SI(Q,\mathcal{C})_\chi$. Since $\alpha\in\img(\phi)$, we have $\alpha=\frac{q(M)}{p(M)}$, for some $M\in\mathcal{I}$. Therefore, $\mathcal{V}(h_{\alpha})\supseteq\overline{\mathcal{O}(M)}$. Since both $\mathcal{V}(h_{\alpha})$ and $\overline{\mathcal{O}(M)}$ are of codimension one in $\mathcal{C}$, we conclude, $\mathcal{V}(h_\alpha)=\overline{\mathcal{O}(M)}$, and hence, $h_{\alpha}$ is irreducible.
\end{proof}
By combining Proposition \ref{irred} with the above remark, we get the following corollary.
\begin{cor}
\phantomsection\label{orbit}
    For every $M\in\mathcal{I}$, $\overline{\mathcal{O}(M)}=\mathcal{V}(h_\alpha)$, where $\alpha=\frac{q(M)}{p(M)}$, i.e. $\alpha\in\img(\phi)$.
\end{cor}

\begin{rmk} \phantomsection\label{RemIrreducibleHomogeneous}
    From the above, every irreducible homogeneous semi-invariant is an associate of some $f_i$ with $1 \le i \le r$ or of some $h_\alpha$ with $\alpha \in \img(\phi)$.
\end{rmk}

\begin{cor}
\phantomsection\label{coprime}
    For a monomial pair $(p,q)$, the monomials $p$ and $q$ are coprime.
\end{cor}
\begin{proof}
    Both $p$ and $q$ are monomials in the $f_i$.  If they are not coprime then
    there exists $f_j$ dividing both.  But then $p/f_j$ and $q/f_j$ would be
    linearly independent monomials of the same weight.  This violates the minimality of $\chi$.   
\end{proof}
\begin{prop}
\phantomsection\label{gen}
    The irreducible elements $f_1,f_2,\dotsc,f_r$ generate 
    $\SI(Q,\mathcal{C})$ as a $\K$-algebra.
\end{prop}
\begin{proof}
    Let $h$ be a non-constant irreducible homogeneous semi-invariant. Consider the vanishing set $\mathcal{V}(h)$. Then, $\mathcal{V}(h)$ is an irreducible closed set of codimension one. Therefore, either $\mathcal{V}(h)=\mathcal{V}(f_i)$, for some $1\leq i \leq r$, or $\mathcal{V}(h)=\overline{\mathcal{O}(M)}=\mathcal{V}(h_\alpha)$, for some $M\in\mathcal{I}$. Using the fact that $h, f_i$ and $h_\alpha$, for $\alpha\in\img(\phi)$, are irreducible and hence prime (since $\mathbb{K}[\mathcal{C}]$ is a UFD), we conclude that in the former case, $h$ and $f_i$ are associates and in the latter case, $h$ and $h_\alpha$ are associates, for some $\alpha\in\img(\phi)$.  Thus either $h$ is a scalar multiple of some $f_i$ or is a linear combination of two monomials in the $f_i$. Therefore, we conclude $f_1,f_2,\dotsc,f_r$ generate $\SI(Q,\mathcal{C})$.
\end{proof}

\subsection{Weight Spaces and Relations}
In this section, we describe the structure of the non-trivial weight spaces and determine the algebraic relations that arise in different weight spaces.
\begin{prop}
\phantomsection\label{uniquechi}
    The minimal weight $\chi$ is the unique minimal weight, with respect to the partial ordering defined on the weights in $\Sigma(Q,\mathcal{C})$, such that $\dim_{\mathbb{K}}\SI(Q,\mathcal{C})_\chi\geq2$.
\end{prop}
\begin{proof}
    Let $\chi'$ be another such minimal weight such that $\dim_{\mathbb{K}}\SI(Q,\mathcal{C})_{\chi'}\geq2$. Choose a monomial pair $(p',q')$ of weight $\chi'$.
    For $\alpha'\in\mathbb{K}$, we define,
    \[
        h'_{\alpha'}=\alpha' p'- q'
    \]
    and a map
    \[
        \begin{array}{cccc}
        \phi': & \mathcal{I} & \longrightarrow & \mathbb{K}^*\\
        \textbf{ } & M & \longmapsto & \alpha'=\frac{q'(M)}{p'(M)}
        \end{array}
    \]
    Arguing as above, we see that if $\alpha'\in\img(\phi')$, then $h'_{\alpha'}$ is irreducible and the vanishing set $\mathcal{V}(h'_{\alpha'})=\overline{\mathcal{O}(M)}$, for some $M\in\mathcal{I}$.
    
    Using Corollary \ref{orbit}, we get,
    $\mathcal{V}(h'_{\alpha'})=\mathcal{V}(h_\alpha)$, for some $\alpha\in\img(\phi)$.
    Therefore, by using the fact that $\mathbb{K}[\mathcal{C}]$ is a UFD and the fact that $h_\alpha$ and $h'_{\alpha'}$ are irreducible, for $\alpha\in\img(\phi)$ and $\alpha'\in\img(\phi')$, respectively, and hence prime, we conclude that $h_\alpha$ and $h'_{\alpha'}$ are associates and hence, $\chi'=\chi$. 
\end{proof}
In fact, we will prove that $\dim_{\mathbb{K}}\SI(Q,\mathcal{C})_\chi=2$, but first we will need some technical lemmas. The first lemma is a well-known observation. \par
\begin{lem}
\phantomsection\label{monoprod}
 Let $Z\in\mathbb{K}[\mathcal{C}]$ with vanishing set
    \[
        \mathcal{V}(Z)=\bigcup_{i=1}^{m}\mathcal{V}(Z_{i})
    \]
    where the $Z_{i}\in\mathbb{K}[\mathcal{C}]$ are non-associate primes. Then $Z=\prod_{i=1}^{m}Z_i^{a_{i}}$, where $a_{i}\in\mathbb{Z}_{>0}$.
\end{lem}
For the rest of this section we fix a monomial pair $(p,q)$ of weight $\chi$. For any weight $\sigma\in\Sigma(Q,\mathcal{C})$, we can write $\sigma$ as $\sigma=n\chi+\delta$, for some maximal $n\in\mathbb{Z}_{\geq0}$ and for some weight $\delta$, and $\sigma\succeq n\chi$. We will refer to $\delta$ as the \emph{remainder} of $\sigma$ and denote it as $\rem{\sigma}$. It follows from Proposition \ref{uniquechi} that  $\dim_{\mathbb{K}}\SI(Q,\mathcal{C})_{\rem{\sigma}}\leq1$.

\begin{lem}
\phantomsection\label{findgamma}
    Let $\sigma\in\Sigma(Q,\mathcal{C})$ be a weight and write
    $\sigma = n\chi + \rem{\sigma}$. Suppose $\SI(Q,\mathcal{C})_{\rem{\sigma}}$ is generated by the monomial $X$. Observe that $\lbrace p^iq^jX\rbrace_{i+j=n}$ are distinct monomials in the $f_i$ in $\SI(Q,\mathcal{C})_{\sigma}$. Assume that $W$ is another monomial in the $f_i$ with $W \in \SI(Q,\mathcal{C})_{\sigma}$ that is not an associate to any of the elements in $\lbrace p^iq^jX\rbrace_{i+j=n}$. For $\mu\in\mathbb{K}^*$, define
    \[
        g_\mu\coloneqq \mu p^nX - W.
    \]
    Then, there exists $\gamma\in\mathbb{K}^*$ such that 
    $h_\alpha$ divides $g_\gamma$
    for some $\alpha\in\img(\phi)$.
\end{lem}
\begin{proof}
    By Remark \ref{RemIrreducibleHomogeneous}, it suffices to find an element $\gamma\in\mathbb{K}^*$ such that there exists an irreducible factor $z$ of $g_{\gamma}$ such that $z$ is not an associate of any $f_{i}$.
    
    If not, then there exists a monomial in the $f_i$, say $f_{i_1}^{s_1}f_{i_{2}}^{s_2}\dotsb f_{i_{t}}^{s_{t}}$ where $i_{k}\in\lbrace1,2,\dotsc,r\rbrace$, which is an associate of $g_{\mu}$ for infinitely many $\mu\in\mathbb{K}^*$. Now choose two such elements $\mu_1,\mu_2\in\mathbb{K}^*$. Therefore, we have
    \[
        \mu_1 p^nX-W=\lambda_1 f_{i_1}^{s_1}f_{i_{2}}^{s_2}\dotsb f_{i_{t}}^{s_{t}}\;\text{and}\;\mu_2 p^nX-W=\lambda_2  f_{i_1}^{s_1}f_{i_{2}}^{s_2}\dotsb f_{i_{t}}^{s_{t}}
    \]
    After multiplying the first equation by $\lambda_2$ and the second equation by $\lambda_1$, and subtracting them, we get 
    $(\lambda_1-\lambda_2)W = (\lambda_1 \mu_2 - \lambda_2\mu_1)p^nX$
    showing that $W$ and $p^n X$ are associates, which contradicts our assumption.
\end{proof}
In the next proposition we calculate the dimension of the weight space $\SI(Q,\mathcal{C})_{\chi}$.
\begin{prop}
\phantomsection\label{2-dim}
    The weight space $\SI(Q,\mathcal{C})_{\chi}$ is $2$ dimensional.
\end{prop}
\begin{proof}
    Let $\lbrace p,q,X_1,X_2,\dotsc,X_t\rbrace$ be all the monomials in the $f_i$ contained in the weight space $\SI(Q,\mathcal{C})_{\chi}$.
    
    Fix a monomial $X_i$, where $i\in\lbrace1,2,\dotsc,t\rbrace$. For $\mu\in\mathbb{K}^*$, define
    \[
        g_{\mu}\coloneqq\mu p-X_i.
    \]
    By Lemma \ref{findgamma}, we know that there exists $\gamma\in\mathbb{K}^*$
    and $\alpha\in\img(\phi)$ such that $h_{\alpha}$ divides $g_{\gamma}$. Since $g_{\gamma}$ and $h_\alpha$ are semi-invariants of the same weight $\chi$, it follows that $g_{\gamma}$ and $h_{\alpha}$ are associates. Recall that $h_\alpha$ is a linear combination of $p$ and $q$. Hence, it follows that $X_i$ can be written as a linear combination of $p$ and $q$.
    By definition, $p$ and $q$ are linearly independent. The statement follows.
\end{proof}
\begin{lem}
\phantomsection\label{algind}
    The monomials $p$ and $q$ are algebraically independent.
\end{lem}
\begin{proof}
    Assume that there is some algebraic dependence relation among $p$ and $q$, which we may assume to be homogeneous. Let this relation be
    \[
        c_p p^m+\sum_{i+j=m}c_{ij}p^iq^j+c_q q^m=0,
    \]
    where $c_p,c_q\in\mathbb{K}^*$, $c_{ij} \in \K$ and $l,i,j,m\in\mathbb{Z}_{\geq0}$. Then we get,
    \[
        c_pp^m+\sum_{i,j}c_{ij}p^iq^j=-c_qq^m
    \]
    which contradicts Corollary \ref{coprime}.
\end{proof}

\begin{prop}
\phantomsection\label{basisnchi}
    Let $\sigma\in\Sigma(Q,\mathcal{C})$ be a weight and write $\sigma=n\chi+\rem{\sigma}$. Then, 
    \[          
        \dim_{\mathbb{K}}\,\SI(Q,\mathcal{C})_{\sigma}=\dim_{\mathbb{K}}\,\SI(Q,\mathcal{C})_{n \chi} = n+1.
    \]
   Moreover, $\lbrace p^iq^jX\rbrace_{i+j=n}$ where $X\in\SI(Q,\mathcal{C})_{\rem{\sigma}}$, forms a basis of $\SI(Q,\mathcal{C})_{\sigma}$.
\end{prop}
\begin{proof}  
    Let the one dimensional space $\SI(Q,\mathcal{C})_{\rem{\sigma}}$ be generated by $X$. Assume that the statement of the proposition does not hold, and let $n$ be minimal such that $\dim_{\mathbb{K}}\,\SI(Q,\mathcal{C})_{\sigma} > n+1$. Clearly, $n > 0$.

    Observe that $\lbrace p^jq^{n-j}X \mid 0 \le j \le n\rbrace$ are linearly independent monomials in the $f_i$ contained in $\SI(Q,\mathcal{C})_{\sigma}$. Let $W$ be another monomial in the $f_i$ in $\SI(Q,\mathcal{C})_{\sigma}$. For $\mu\in\mathbb{K}^*$, define
    \[
        g_\mu\coloneqq\mu p^nX-W.
    \]
    Using Lemma \ref{findgamma}, we can find some $\gamma\in\mathbb{K}^*$ such that $g_{\gamma}=h_{\alpha}X'$, for some $\alpha\in\img(\phi)$ and $X'\in\SI(Q,\mathcal{C})_{(n-1)\chi + \rem{\sigma}}$. By minimality of $n$, it follows that $X'$ is a linear combination of the $\{p^iq^{n-1-i}X\}$. This, together with the fact that $h_\alpha$ is a linear combination of $p,q$ yields that $W$ is a linear combination of the $\lbrace p^iq^{n-i}X\rbrace$. This proves the proposition.
\end{proof}
\begin{prop}
\phantomsection\label{TrinomialRelation}
    Suppose the map $\phi$ is not surjective. If $\gamma\in\mathbb{K}^*\setminus\img(\phi)$, then $h_\gamma$ gives rise to a non-trivial linear dependence relation between three distinct pairwise coprime monomials in the $f_i$ of weight $\chi$.
\end{prop}
\begin{proof}
Assume first that $h_\gamma$ is divisible by some $h_\alpha$ where $\alpha \in \img{\phi}$. Since $h_\gamma, h_\alpha$ have the same weight, that means that they are associate. Hence, there exists $\lambda \in \K$ such that $h_\gamma = \lambda h_\alpha$ which gives a linear dependence relation between $p,q$, since $\alpha \ne \gamma$. This is a contradiction. It follows from Remark \ref{RemIrreducibleHomogeneous} that $h_\gamma = \lambda Z$ for some monomial $Z$ in the $f_i$, which cannot be in $\{p, q\}$. We get the relation $\lambda Z - \gamma p + q = 0$. By Corollary \ref{coprime}, it follows that $p,q,Z$ are pairwise coprime. This proves the proposition.
\end{proof}

Notice that $\img(\phi)$ is co-finite in $\K$. To see this note that $\phi$ can be naturally defined as a morphism $\phi: U \to \mathbb{P}^1$ with $\phi(M) = [p(M): q(M)]$ and, since $U$ is irreducible, so is $\img(\phi)$. Since, clearly, the latter is not a single point, this implies the claim. In particular, the number of relations provided by Proposition~\ref{TrinomialRelation} is finite. We let $\mathbb{K}^*\setminus\img(\phi) = \{\gamma_1, \ldots, \gamma_m\}$. For each $1 \le i \le m$, we let $H_i$ denote the relation of the form $\lambda_i Z_i - \gamma_i p + q$ obtained in the proof of Proposition \ref{TrinomialRelation}, where $\lambda_i$ is some scalar and $Z_i$ some monomial of weight $\chi$ in the $f_i$. Note that $\{p, q, Z_1, \ldots, Z_m\}$ are pairwise coprime. In particular, $m+2 \le r$.

\begin{prop}
\phantomsection\label{idealofreln}
    The ideal of relations among the $f_i$ is generated by the $m$ relations $H_1, \ldots, H_m$.
\end{prop}

\begin{proof}
 Consider the decomposition of $\sigma$ as $\sigma = n\chi + \rem{\sigma}$. Recall that $\SI(Q,\mathcal{C})_{\sigma}$ has a basis given by $\{ p^iq^jX\}_{i+j=n}$. Let $Y$ be any monomial in the $f_i$ of weight $\sigma$ which is not in this basis. It follows from Proposition~\ref{basisnchi} that $Y = XY'$ where $Y'$ is some homogeneous polynomial $P(p,q)$ of degree $n$ in $p,q$. Then $P$ factors as $P=\lambda p^a \prod_{\mu\in A} h_\mu$ for some $a$ and some scalar $\lambda$.

Writing $Y'$ as $Y' = \prod_{i=1}^r f_i^{a_i}$
we have $\prod_{i=1}^r f_i^{a_i} = \lambda p^a \prod_{\mu\in A} h_\mu$.
By the unique factorization, this implies that
$h_\mu$ is, up to a scalar, a monomial in the $f_i$ for each $\mu \in A$.
Every such $\mu$ cannot lie in $\img(\phi)$ since none of the $f_i$ vanish 
anywhere on $U$.
\end{proof}

\subsection{Main Result}
Now that we have determined the generators, the structure of the non-trivial weight spaces and the type of algebraic relations that arise, we are ready to prove the main result.
\begin{thm}
\phantomsection\label{mainthm}
    The semi-invariant ring $\SI(Q,\mathcal{C})$ is a complete intersection of codimension $\#\lbrace\text{monomials of weight } \chi\rbrace-2$.
\end{thm}
\begin{proof}
    Since the semi-invariant ring $\SI(Q,\mathcal{C})$ is a sub-ring of $\mathbb{K}[\mathcal{C}]$, which is an integral domain, it follows that $\SI(Q,\mathcal{C})$ is an integral domain. In particular, the ideal of relations, call it $J$, is a prime ideal.

    From above discussion and Proposition~\ref{idealofreln}, we know that
    the relations $H_1,\dots,H_m$ generate $J$.

    Consider the Jacobian $\mathcal{J}=\left(\frac{\partial H_i}{\partial{f_j}}\right)_{\substack{1\leq i \leq m \\ 1 \leq j \leq r}}$ which is of size $m\times r$. Since $\{p,q, X_1, \ldots, X_m\}$ are pairwise coprime, we see that this matrix has rank $m$.
    Hence, we conclude that $\SI(Q,\mathcal{C})$ is a complete intersection of codimension $\#\lbrace\text{monomials of weight } \chi\rbrace-2$.
\end{proof}
\section{Examples}
In the previous section, we saw that under the UFD assumption and when the maximal orbits have codimension one, the semi-invariant ring $\SI(Q,\mathcal{C})$ is a complete intersection and is not multiplicity free. However, just having an irreducible component with maximal orbits of codimension one might not give us this. To illustrate this, we present the following example.
\subsection{Example 1:}
\phantomsection\label{ex1}
Consider the algebra given by the following quiver $Q$,
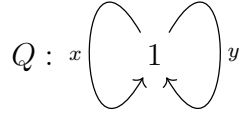
\[
    Q:
    \begin{tikzcd}
	   1 \arrow[out=120, in=240,loop,swap,"x"] \arrow[out=60,in=300,loop, "y"]
    \end{tikzcd}
\]
with $R = \langle x^2, y^2, xy, yx \rangle$, and consider the dimension vector $\beta=(2)$. The algebra $A = \K Q/R$ is a string algebra. Notice that $\rep_{\beta}(Q,R)$ has an irreducible component $\mathcal{C}$ containing the following family of non-isomorphic band modules $\{M_{\lambda}\}_{\lambda\not=0}$
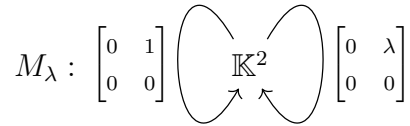
\[
    M_{\lambda}:
    \begin{tikzcd}
        \mathbb{K}^2 \arrow[out=120, in=240,loop,swap,"{\spmat{0 & 1\\ 0 & 0}}"] \arrow[out=60,in=300,loop, "{\spmat{0 & \lambda\\ 0 & 0}}"]
    \end{tikzcd}
\]
It is easy to check that the coordinate ring $\mathbb{K}[\mathcal{C}]$ is not a UFD (or we can also use Remark \ref{lrmk} (\ref{notUFD})). Note that string algebras are tame. Therefore, it follows that $\mathrm{codim}\,\mathcal{O}(M_{\lambda})=1$ for $\lambda\not=0$.

Since the algebra given by the above quiver with relations is a local algebra, by using King's criterion (see \parencite{King} or \parencite[Proposition 9.8.3]{weyderk}) we conclude that $\mathcal{C}$ cannot have any $\theta$-semi-stable point for any non-trivial stability parameter $\theta\in\mathbb{Z}^{Q_{0}}$. Therefore, the semi-invariant functions on $\mathcal{C}$ are the constant functions, i.e., the semi-invariant ring $\SI(Q,\mathcal{C})$ is trivial.

The subsequent two examples showcase quivers with relations that not only fulfill the hypothesis of our main result but also yield semi-invariant rings that are complete intersection but not polynomial rings. Moreover, the codimension of the complete intersection can be arbitrarily large.
\subsection{Example 2:}
\phantomsection\label{ex2}
Consider the algebra $A$ given by the following quiver $Q$,
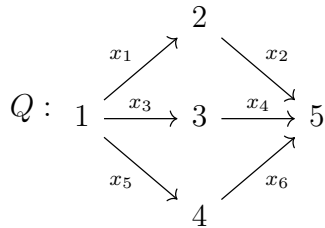
\[
    Q:
    \begin{tikzcd}
    \, & 2 \arrow[rd, "x_2"] \\
    1 \arrow[ur, "x_1"] \arrow[r, "x_3"] \arrow[rd,swap, "x_5"] & 3 \arrow[r, "x_4"] & 5 \\
    \, & 4 \arrow[ur, swap, "x_6"]
    \end{tikzcd}
\]
with the relation $x_1x_2+x_3x_4+x_5x_6=0$ and consider the dimension vector $\beta=(1,1,1,1,1)$. The following is a family of pairwise non-isomorphic representations $\lbrace M_\lambda\rbrace_{\lambda\not=0,-1}$
\[
    M_\lambda:
    \begin{tikzcd}
    \, & \mathbb{K} \arrow[rd, "1"] \\
    \mathbb{K} \arrow[ur, "1"] \arrow[r, "1"] \arrow[rd,swap, "1"] & \mathbb{K} \arrow[r, "\lambda"] & \mathbb{K} \\
    \, & \mathbb{K} \arrow[ur, swap, "-1-\lambda"]
    \end{tikzcd}
\]
We have
\[
    \mathbb{K}[\rep_{\beta}(Q,R)]\cong\mathbb{K}[x_1,x_2,\dotsc,x_6]/(x_1x_2+x_3x_4+x_5x_6).
\]
In order to understand the latter, let us consider the quiver $Q'$ of type $\widetilde{\mathbb{D}}_4$ with the dimension vector $\mathbf{d}=(1,1,1,1,2)$ as follows,
\[
    Q':
    \begin{tikzcd}
    1 \arrow[dr,"a_1"] && 3\\
     & 5 \arrow[ur,"a_3"] \arrow[dr,swap,"a_4"] \\
    2 \arrow[ur,swap,"a_2"] && 4
    \end{tikzcd}
\]
By the calculation in \parencite[Subsection 10.10.2]{weyderk}, we know that $\mathbb{K}[\rep_{\beta}(Q,R)]\cong\SI(Q',\mathbf{d})$. Therefore, $\mathbb{K}[\rep_{\beta}(Q,R)]$ is a UFD and hence $\C:=\rep_{\beta}(Q,R)\cong \mathcal{V}(x_1x_2+x_3x_4+x_5x_6)$ is an irreducible hypersurface in $\mathbb{A}^6$. 

Notice that each $M_\lambda$ is a brick (i.e., has a one dimensional endomorphism algebra), hence $\dim \mathcal{O}(M_{\lambda})=\dim\GL_\beta -1 = 4$. These $M_\lambda$ lie on orbits of maximal dimension in $\C$. Clearly, $\C$ has dimension $5$, so $\C$ has an open set consisting of infinitely many codimension one orbits.

Since the dimension vector for $Q$ is chosen to be $\beta=(1,1,1,1,1)$, we get
\[
    \SI(Q,\C)=\mathbb{K}[\rep_{\beta}(Q,R)]^{\SL_{\beta}}=\mathbb{K}[\rep_{\beta}(Q,R)]=\mathbb{K}[x_1,x_2,\dotsc,x_6]/(x_1x_2+x_3x_4+x_5x_6) 
\]
Therefore, $\SI(Q, \C)$ is a hypersurface, and in particular a complete intersection.

\medskip

Next, we would like to generalize the last example to show that the number of relations occurring in the weight space can be arbitrary. The next example provides an example where the ring of semi-invariants is a complete intersection but not a hypersurface.

\subsection{Example 3:}
Consider the following quiver $Q_{n}$, $n\geq 2$,
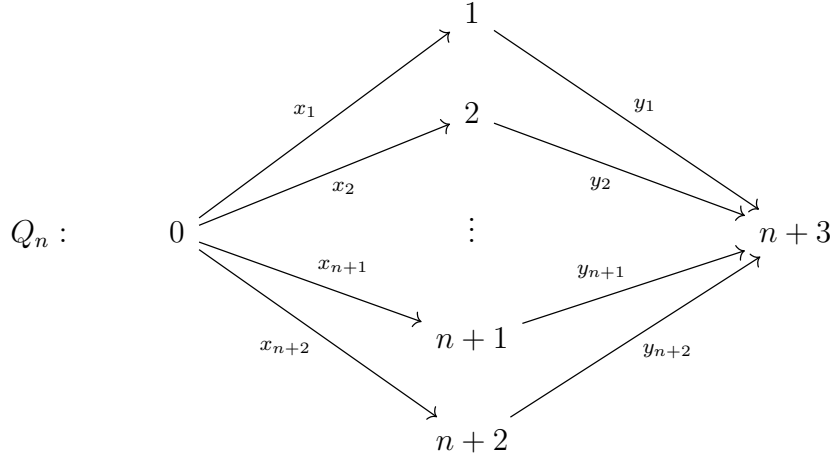
\[\begin{tikzcd}
	&&&& 1 \\
	&&&& 2 \\
	{Q_n:} & 0 &&& \vdots &&& {n+3} \\
	&&&& {n+1} \\
	&&&& {n+2}
	\arrow["{x_1}", from=3-2, to=1-5]
	\arrow["{x_2}"', from=3-2, to=2-5]
	\arrow["{x_{n+1}}", from=3-2, to=4-5]
	\arrow["{x_{n+2}}"', from=3-2, to=5-5]
	\arrow["{y_1}", from=1-5, to=3-8]
	\arrow["{y_2}"', from=2-5, to=3-8]
	\arrow["{y_{n+1}}", from=4-5, to=3-8]
	\arrow["{y_{n+2}}"', from=5-5, to=3-8]
\end{tikzcd}\]
with the relations $\lbrace x_1y_1+kx_2y_2+x_{k+2}y_{k+2}\rbrace_{k=1}^n$ and dimension vector $\beta=(1,1,\dotsc,1)$. We have a following family of non-isomorphic representations, $\lbrace M_\lambda\rbrace_{\lambda\in\mathcal{I}}$, where $$\mathcal{I}=\lbrace \lambda\in\mathbb{K}^*\mid\lambda\not=-\frac{1}{k},1\leq k\leq n\rbrace,
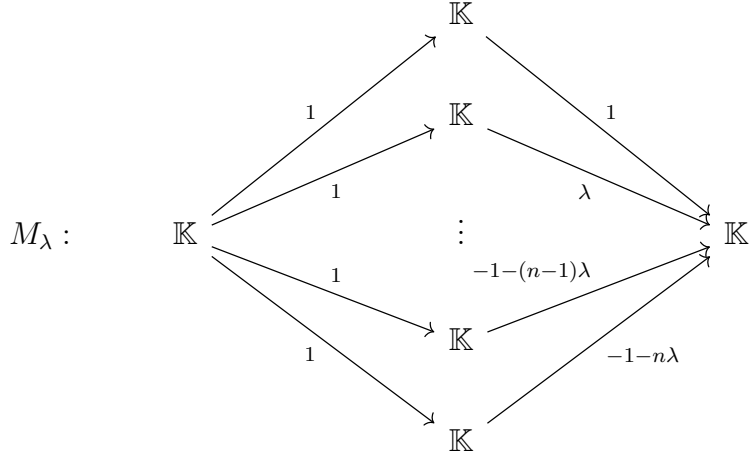
$$
\[\begin{tikzcd}
	&&&& {\mathbb{K}} \\
	&&&& {\mathbb{K}} \\
	{M_\lambda:} & {\mathbb{K}} &&& \vdots &&& {\mathbb{K}} \\
	&&&& {\mathbb{K}} \\
	&&&& {\mathbb{K}}
	\arrow["1", from=3-2, to=1-5]
	\arrow["1"', from=3-2, to=2-5]
	\arrow["1", from=3-2, to=4-5]
	\arrow["1"', from=3-2, to=5-5]
	\arrow["1", from=1-5, to=3-8]
	\arrow["\lambda"', from=2-5, to=3-8]
	\arrow["{-1-(n-1)\lambda}", from=4-5, to=3-8]
	\arrow["{-1-n\lambda}"', from=5-5, to=3-8]
\end{tikzcd}\]

Consider the polynomial ring $\mathbb{K}[x_i, y_i \mid 1 \le i \le n+2]$ and the ideal of this ring generated by the above given relations,
\[
    I_{n}=(x_1y_1+kx_2y_2+x_{k+2}y_{k+2} \mid 1 \leq k \leq n)
\]
Then, we have,
\[
    \mathbb{K}[\rep_{\beta}(Q_n,R)]=\mathbb{K}[x_i, y_i \mid 1 \le i \le n+2]/I_n
\]
\begin{prop}
\phantomsection\label{intdomt_n}
    The ring $T_n\coloneqq\mathbb{K}[x_i, y_i \mid 1 \le i \le n+2]/I_n$ is an integral domain.
\end{prop}
\begin{proof}
    Make a change of coordinates as follows:
    \begin{align*}
        &x_{i}=z_{2i-1}-z_{2i}\\
        &y_{i}=z_{2i-1}+z_{2i}
    \end{align*}
    for $1\leq i\leq (n+2)$. Then, for $1\leq k\leq n$,
    \[
        x_1y_1+kx_2y_2+x_{k+2}y_{k+2}= z_1^2-z_2^2+kz_3^2-kz_4^2+z_{2k+3}^2-z_{2k+4}^2
    \]
    and let
    \[
        p_k\coloneqq z_2^2-z_1^2+kz_4^2-kz_3^2+z_{2k+4}^2.
    \]
    Therefore, $I_n=(z_5^2-p_1,z_7^2-p_2,\dotsc,z_{2n+3}^2-p_n)$, and hence, $T_n=\mathbb{K}[z_1,z_2,\dotsc,z_{2n+4}]/I_n$. We first note that the $p_k$ are irreducible as quadratic polynomials; this may be confirmed by direct computations. It follows from this that each $p_k$ and each ratio $\frac{p_k}{p_{k+1}}$ is a non-square in $L\coloneqq\mathbb{K}(z_1,z_2,\dotsc,z_{2n+4})$.

    By induction, it is not hard to prove that $\sqrt{p_i}\notin L(\sqrt{p_1},\sqrt{p_2},\dotsc,\sqrt{p_{i-1}})$, which implies that
    \[
         [L(\sqrt{p_1},\dotsc,\sqrt{p_n}):L]=2^n
    \]
    Therefore, $\left\{\sqrt{p_1^{i_1}p_2^{i_2}\dotsb p_n^{i_n}} \mid i_{l}\in\lbrace 0,1\rbrace\right\}$ is an  $L-$basis
    for $L(\sqrt{p_1},\dotsc,\sqrt{p_n})$.

    Let $S\coloneqq\mathbb{K}[z_1,z_2,z_3,z_4,z_6,z_8,\dotsc,z_{2n+4}][z_5,z_7,\dotsc,z_{2n+3}]$. Now, define
    \[
        \begin{array}{cccc}
            \phi: & S & \rightarrow & L(\sqrt{p_1},\sqrt{p_2},\dotsc,\sqrt{p_n})\\
            \, & f & \mapsto & f(z_1,z_2,z_3,z_4,\sqrt{p_1},z_6,\sqrt{p_2},z_8,\dotsc,\sqrt{p_n},z_{2(n+2)}) 
        \end{array}
    \]
    Clearly, $I_n\subseteq \mathrm{ker}(\phi)$. Let $g\in \mathrm{ker}(\phi)$, then successively dividing $g$ by $z_{2k+3}^2-p_k$, for $1\leq k \leq n$, we can write $g$ as,
    \[
        g=\sum_{0\leq i_1,i_2,\dotsc,i_n\leq 1}Q_{i_1i_2\dotsc i_n}z_5^{i_1}z_7^{i_2}\dotsb z_{2n+3}^{i_n} + h,
    \]
    where $h\in I_n$ and $Q_{i_1i_2\dotsc i_n}\in\mathbb{K}[z_1,z_2,z_3,z_4,z_6,z_8,\dotsc,z_{2(n+2)}]$.
    
    Now,
    \[
        \phi(g)=\sum_{0\leq i_1,i_2,\dotsc,i_n\leq 1}Q_{i_1i_2\dotsc i_n}\sqrt{p_1^{i_1}p_2^{i_2}\dotsb p_n^{i_n}}=0
    \]
    Therefore, each $Q_{i_1i_2\dotsc i_n}=0$ and hence, $g\in I_n$. So, $\mathrm{ker}(\phi)=I_n$.\par
    Therefore, we have an injective ring homomorphism,
    \[
        T_n\hookrightarrow L(\sqrt{p_1},\sqrt{p_2},\dotsc,\sqrt{p_n})
    \]
    Hence, $T_n$ is an integral domain.
\end{proof}
\begin{lem}
\phantomsection\label{primen}
    The element $\overline{x}_1\in T_n$ is a prime element.
\end{lem}
\begin{proof}
    Let $J_n\coloneqq(kx_2y_2+x_{k+2}y_{k+2} \mid 1 \leq k \leq n)$. Consider the ideal $(\overline{x}_1)\subseteq T_n$ generated by $\overline{x}_1$.
    \begin{align*}
        T_n/(\overline{x}_1)&\cong\mathbb{K}[x_1,y_1,x_2,\dotsc,y_{n+2}]/(I_n,x_1)\\
        &\cong\mathbb{K}[x_2,y_2,x_3,\dotsc,y_{n+2},y_1]/J_n
    \end{align*}
    Let $S\coloneqq\mathbb{K}[x_2,y_2,x_3,\dotsc,y_{n+2}]/J_n$. Using a similar argument as in the proof of Proposition~\ref{intdomt_n}, we conclude $S$ is an integral domain. Therefore,
    \[
        T_n/(\overline{x}_1)\cong S[y_1]
    \]
    is an integral domain, and hence, $\overline{x}_1\in T_n$ is a prime element.
\end{proof}
\begin{prop}
    The ring $T_n$ is a UFD.
\end{prop}
\begin{proof}
    We will use induction on $n$ and Nagata's criterion for UFD (refer to \parencite[Lemma 2]{Nagata}). From Example~\ref{ex1}, we know that the statement is true for $n=1$.

    Assume the statement is true for all $k<n$. Now, we'll prove the statement for $k=n$.

    Subtracting the first relation from any other defining relation in $I_n$  gives us $kx_2y_2-x_3y_3+x_{k+3}y_{k+3}$ for $1 \leq k \leq n-1$. Now, let
    $A_{n-1}\coloneqq(kx_2y_2-x_3y_3+x_{k+3}y_{k+3} \mid 1 \leq k \leq n-1)$. By Proposition~\ref{intdomt_n}, we know that $T_n$ is an integral domain and by Lemma~\ref{primen}, we know that $\overline{x}_1\in T_n$ is a prime element. Now, localize $T_n$ at $\overline{x}_1$. Then,
    \begin{align*}
        (T_n)_{\overline{x}_1}\cong\mathbb{K}[x_2,y_2,x_3\dotsc,y_{n+2},x_1,x_1^{-1}]/A_{n-1} \cong T[x_1,x_1^{-1}]
    \end{align*}
    where $T\coloneqq\mathbb{K}[x_2,y_2,x_3,\dotsc,y_{n+2}]/A_{n-1}$.
    Notice that $T\cong T_{n-1}$ which by the induction hypothesis is a UFD. Thus the localization of $T_n$ at the prime element $\overline{x}_1$ 
    is a UFD. By Nagata's criterion, we conclude that $T_n$ is a UFD.
\end{proof}

By Proposition~\ref{intdomt_n}, we get,
\[
    \rep_{\beta}(Q_n,R)=\mathcal{V}(I_n)
\]
is irreducible in $\mathbb{A}^{2n+4}$. Since the dimension vector for $Q_n$ is chosen to be $\beta=(1,1,\dotsc,1)$, we get,
\[
    \SI(Q_n,\beta)=\mathbb{K}[\rep_{\beta}(Q_n,R)]^{\SL_{\beta}}=T_n 
\]
Therefore, by the main result, we know that $\SI(Q_n,\beta)$ is a complete intersection of codimension $n$.
\begin{rmk}
    Notice that the coefficients $k$ in front of $x_2y_2$ in the defining generators of $I_n$ in the above example are needed for the ring $T_n$ to be a UFD, for $n\geq 2$. If we replace these coefficients by $1$, the ideal $A_{n-1}$ in the above proof will be generated by binomials and the induction will not be valid as $T_2$ would be $\mathbb{K}[x_1, x_2, x_3, x_4, y_1, y_2, y_3, y_4]/ ( x_1y_1 + x_2y_2 + x_3y_3, x_4y_4 - x_3y_3)$ which is not a UFD.
\end{rmk}

\section{The Case of Hereditary Algebras}
In this section we examine the situation in the context of hereditary algebras.
Let $Q=(Q_0,Q_1)$ be an acyclic quiver and $\beta$ be a dimension vector. Notice that the representation space $\rep_{\beta}(Q)$ is an affine space and hence, irreducible. Also, the coordinate ring $\mathbb{K}[\rep_{\beta}(Q)]$ is a polynomial ring, hence is a UFD. We will need few definitions before we get to the main result of this section. We will denote the dimension vector of a representation $V$ of $Q$ by $\underline{\dim}V$.
\begin{defn}
    A representation $V\in\rep_{\beta}(Q)$ is said to be a $\emph{general representation}$ (with respect to a property $\mathcal{P}$) if there exists an open (dense) subset $U\subseteq\rep_{\beta}(Q)$ such that $V\in U$ and every representation in $U$ satisfies the property $\mathcal{P}$.
\end{defn}
\begin{defn}
    A representation $V\in \rep_{\beta}(Q)$ is called a $\emph{Schur representation}$ (or a $\emph{brick}$) if $\mathrm{End}(V)$ is a division algebra (or equivalently, coincides with the scalar multiples of the identity, since $\mathbb{K}$ is assumed to be algebraically closed). A dimension vector $\beta$ is called a $\emph{Schur root}$ if a general representation in $\rep_{\beta}(Q)$ is a brick, i.e. there is an open non-empty set $U\subseteq\rep_{\beta}(Q)$ such that any representation in $U$ is a brick.
\end{defn}
For a quiver $Q=(Q_0,Q_1)$, we define the Euler form $\langle\cdot,\cdot\rangle$ on the Grothendieck group $K_{0}(\mathbb{K}Q)$, which can be identified with $\mathbb{Z}^{Q_0}$, in the following way:
\[
    \langle\alpha,\beta\rangle\coloneqq\sum_{x\in Q_0}\alpha(x)\beta(x)-\sum_{a\in Q_1}\alpha(ta)\beta(ha),\;\text{where}\;\alpha,\beta\in\mathbb{Z}^{Q_0}
\]
The Euler characteristics of two representations $V,W$ of $Q$ is defined to be
\[
    \chi(V,W)\coloneqq\dim \Hom(V,W)-\dim\Ext^1(V,W)
\]
By \parencite[Proposition 2.5.2]{weyderk}, we know that the above Euler form coincides with the Euler characteristics in the following way: if $V,W$ are two representations of $Q$ then
\[
    \langle\underline{\dim}V,\underline{\dim}W\rangle=\dim \Hom(V,W)-\dim\Ext^1(V,W)=\chi(V,W)
\]

Recall that if $\beta$ is a Schur root, then we call $\beta$ a $\emph{real}$ Schur root if $\langle\beta,\beta\rangle=1$, we call $\beta$ an $\emph{isotropic}$ Schur root if $\langle\beta,\beta\rangle=0$, and we call $\beta$ an $\emph{imaginary}$ Schur root if $\langle\beta,\beta\rangle<0$.

We can also study the decomposition of a representation of dimension $\beta$ into its indecomposable representations. V. Kac defined the canonical decomposition of a dimension vector and showed that if $V\in\rep_{\beta}(Q)$ is within an open (dense) set, then the dimension vector of its indecomposable direct summands are independent of $V$ \parencite[see page 150]{KAC}.

\begin{defn}
    Let $\beta$ be a dimension vector. If $\beta=\beta_1+\beta_2+\dotsc+\beta_t$ is such that all $\beta_i$ are Schur roots and a general representation $V$ of dimension $\beta$ decomposes into its indecomposable direct summands as $V=V_1\oplus V_2\oplus\dotsc\oplus V_t$ where $\underline{\dim}V_i=\beta_i$, then we write
    \[
        \beta= \beta_1\oplus\beta_2\oplus\dotsc\oplus\beta_t
    \]
    and is called the \emph{canonical decomposition} of $\beta$.
\end{defn}
\begin{rmk}
    The canonical decomposition of a dimension vector always exists and is unique (up to rearrangements), see \parencite[page 150]{KAC}.
    
    Let $\mathrm{ext}(\beta_i,\beta_j)$ denote the generic value of $\Ext^1(V_i,V_j)$ on $\rep_{\beta_i}(Q)\times\rep_{\beta_j}(Q)$. For a dimension vector $\beta$, the decomposition $\beta=\beta_1+\beta_2+\dotsc+\beta_t$ yields the canonical decomposition if and only if all $\beta_i$ are Schur roots and $\mathrm{ext}(\beta_i,\beta_j)=0$ for all $i\not=j$ (refer to \parencite[Proposition 3]{KAC} or \parencite[Theorem 11.2.1]{weyderk}).
\end{rmk}

Recall that a dimension vector $\beta$ is called $\emph{prehomogeneous}$ if $\rep_{\beta}(Q)$ contains an open (dense) $\GL_{\beta}$-orbit. In terms of canonical decomposition this means that a dimension vector is prehomogeneous if and only if its canonical decomposition only contains real Schur roots.
\begin{defn}
We call a dimension vector $\beta$ $\emph{almost prehomogeneous}$ if $\rep_{\beta}(Q)$ has an open non-empty set which is a union of codimension one $\GL_{\beta}$-orbits.
\end{defn}

Recall that if $V\in\rep_{\beta}(Q)$, then the $\GL_{\beta}$-orbit of $V$ has codimension equal to $\dim\Ext^1(V,V)$. Hence, we see that $\beta$ is almost prehomogeneous exactly when a general representation $V$ of dimension vector $\beta$ has $\Ext^1(V,V)$ one dimensional.

The proposition below gives a characterisation of almost prehomogeneous dimension vectors and gives an interpretation of such a dimension vector in terms of canonical decomposition.
\begin{prop}
\phantomsection\label{alprehom} 
    Let $Q$ be an acyclic quiver and $\beta$ be a dimension vector. The following statements are equivalent:
    \begin{enumerate}
        \item A general module $M$ decomposes as $M=M_1\oplus M_2$, where 
        \[
            \Ext^1(M,M)=\mathbb{K}=\Ext^1(M_2,M_2)
        \]
        \item The representation space $\rep_{\beta}(Q)$ consists of an open set which is an union of codimension one $\GL_{\beta}$-orbits.
        \item The dimension vector $\beta$ is almost prehomogeneous.
        \item If $\beta= \beta_1\oplus\beta_2\oplus\dotsb\oplus\beta_t$ is the canonical decomposition of $\beta$, then for exactly one $j_0$, $\beta_{j_0}$ is an isotropic Schur root and $\beta_i$ is a real Schur root, for all $i \neq j_0$.
    \end{enumerate}
\end{prop}
\begin{proof}
    $(\emph{1.}\Longrightarrow\emph{2.})$ Let $U$ be the union of all codimension 1 orbits. Then, $U$ is non-empty by the statement $\emph{1.}$ Let
    \[
        \begin{array}{cccc}
             \Theta: & \rep_{\beta}(Q) & \longrightarrow & \mathbb{N}\\
            \, & M & \longmapsto & \mathrm{codim}\,\mathcal{O}(M)
        \end{array}
    \]
    We know that the map $\Theta$ is lower semi-continuous. Hence, $\Theta^{-1}([0,1])$ is an open set in $\rep_{\beta}(Q)$. Notice that $\Theta^{-1}(0)=\emptyset$. Therefore, we conclude, $\Theta^{-1}(1)=U$ is an open set.
    
    $(\emph{2.}\Longrightarrow\emph{3.})$ This follows from the definition.
    
    $(\emph{3.}\Longrightarrow\emph{4.})$ Since $\beta$ is almost prehomogeneous, there exists an open set, $U$, which is a union of codimension 1 $\GL_{\beta}$-orbits. Let
    \[
        \beta= \beta_1\oplus\beta_2\oplus\dotsb\oplus\beta_t
    \]
    be the canonical decomposition of $\beta$. Any general representation $N$ in this open set will decompose as
    \[
        N= N_1\oplus N_2\oplus\dotsb\oplus N_t
    \]
    where $N_i$ is an indecomposable representation with dimension vector $\beta_i$ and $\Ext^1(N_j,N_k)=0$, for $j\not=k$. Since $\mathrm{codim}\,\mathcal{O}(N)=1$, we have, $\Ext^1(N,N)=\mathbb{K}$. Using the bilinearity of $\Ext^1$, we conclude that for exactly one $1\leq j\leq t$, $\Ext^1(N_j,N_j)=\mathbb{K}$, and $\Ext^1(N_i,N_i)=0$, for all $1\leq i\leq t$ such that $i\not=j$. Since each $N_i$ is indecomposable, we conclude that exactly one $\beta_j$, for $1\leq j\leq t$ is an isotropic Schur root and all other $\beta_i$ are real Schur root, for $1\leq i\leq t$ such that $i\not=j$.
    
    $(\emph{4.}\Longrightarrow\emph{1.})$ This follows from the properties of canonical decomposition and \parencite[Proposition 11.3.7]{weyderk}.
\end{proof}
\begin{rmk}
    If in the statement $(1)$ of above proposition, $M_2$ is chosen to be minimal with respect to the property $\Ext^1(M_2,M_2)=\mathbb{K}$, then $M_2$ is a brick. This is due to the following reason-- using \parencite[Proposition 11.3.7]{weyderk} and the fact that $M_2$ is minimal with respect to the property $\Ext^{1}(M_2,M_2)=\mathbb{K}$, we conclude that the canonical decomposition of $\underline{\dim}M_2$ is a single isotropic Schur root. Therefore, by using the Euler form, we get
    \[
        0=\langle\underline{\dim} M_2,\underline{\dim} M_2\rangle=\dim\Hom(M_2,M_2)-\dim\Ext^{1}(M_2,M_2)=\dim\mathrm{End}(M_2)-1
    \]
    Hence, we get that $\mathrm{End}(M_2)\cong\mathbb{K}$, i.e., $M_2$ is a brick.
\end{rmk}

The following theorem was first proven in \parencite{weypaq} for the case of an isotropic Schur root. Specifically, when the canonical decomposition involves both real Schur roots and a single isotropic Schur root, it is possible to simplify the problem by considering a smaller quiver. This reduction can be achieved by applying Schofield's reduction theorem (refer to \parencite[Theorem 11.4.9]{weyderk}), which allows one to remove the real Schur roots from the beginning and from the end of the canonical decomposition. Consequently, the problem transforms into computing the ring of semi-invariants for a single isotropic Schur root over a smaller quiver.

In \parencite{weypaq}, it was proven that the ring of semi-invariants is a hypersurface. However, considering that our techniques extend beyond the hereditary case and building upon the examples presented in the previous section, the most favorable outcome we can obtain is a complete intersection.

\begin{thm}
\phantomsection\label{acycmainthm}
    Let $Q$ be an acyclic quiver and $\beta$ be a dimension vector whose canonical decomposition contains exactly one isotropic Schur root and all others are real Schur root. Then, the semi-invariant ring $\SI(Q,\beta)$ is a complete intersection.
\end{thm}
\begin{proof}
    This follows from Proposition \ref{alprehom} and Theorem \ref{mainthm}.
\end{proof}

\section{The General Case}
In the previous sections, we have dealt with the case where the irreducible component $\mathcal{C}$ has orbits of codimension at most one, under the UFD assumption. In the case of an open orbit, we get a polynomial ring for the ring of semi-invariants; and in case where there is no open orbit but codimension one orbits, we get a complete intersection for the ring of semi-invariants. When the maximal orbit has codimension at least two, it seems difficult to get structural results for the ring of semi-invariants. Nevertheless, under the UFD assumption, we can still draw some conclusion about the semi-invariant ring.

As before, we assume that $A$ denotes a finite dimensional (associative and unital) algebra over an algebraically closed field $\mathbb{K}$ of characteristics $0$. Let $(Q,R)$ be a quiver with relations associated to the algebra $A$, where $R\subseteq\mathbb{K}Q$ is an admissible ideal, and we let $\beta$ be a dimension vector. In order to get to the main result of this section, we will need a proposition, which is well known to experts.
\begin{prop}
\phantomsection\label{ratSI}
    Let $\mathcal{C}\subseteq \rep_{\beta}(Q,R)$ be an irreducible component such that the coordinate ring $\mathbb{K}[\mathcal{C}]$ is a UFD. If there exists a non-constant $\GL_{\beta}$-invariant rational function $r$ on $\mathcal{C}$, then there exists homogeneous $p,q\in\SI(Q,\mathcal{C})$ of the same weight such that $r = p/q$.
\end{prop}
\begin{proof}
    We let $r = p/q$ where we may assume that $p$ and $q$ are coprime. For any element $g\in\GL_{\beta}$, we have 
    \begin{equation}
    \phantomsection\label{inveq}
        (g\cdot p)q=(g\cdot q)p.
    \end{equation}
    Therefore, we get $p\mid g\cdot p$ and $q\mid g\cdot q$. Hence, there exists $f_1,f_2\in\mathbb{K}[\mathcal{C}]$ such that
    \[
        g\cdot p=pf_1\text{ and }g\cdot q=qf_2.
    \]
    By Equation \ref{inveq}, we conclude that $f_1=f_2$. Define $f_g\coloneqq f_1=f_2$. By similar reasoning as above, there exists $f_{g^{-1}}\in\mathbb{K}[\mathcal{C}]$ such that $g^{-1}\cdot p=pf_{g^{-1}}$ and $g^{-1}\cdot q=qf_{g^{-1}}$. Therefore, we get
    \[
        p=e\cdot p=(gg^{-1})\cdot p=g\cdot(g^{-1}\cdot p)=g\cdot(pf_{g^{-1}})=pf_{g}(g\cdot f_{g^{-1}})
    \]
    Hence, we get $f_{g}(g\cdot f_{g^{-1}})=1$, i.e., $f_{g}$ is a unit. By Lemma \ref{units}, we know that $f_g\in\mathbb{K}^*$. Therefore, we conclude $p$ and $g\cdot p$ are associates for every $g\in\GL_{\beta}$. Now, by similar argument as in Proposition \ref{genSI}, we conclude that $p$ is a homogeneous semi-invariant. Similarly, we prove that $q$ is a homogeneous semi-invariant. Clearly, since $p/q$ is an invariant function, the semi-invariants $p,q$ need to have the same weight.
\end{proof}

\begin{defn}\,
    \begin{enumerate}
        \item The semi-invariant ring $\SI(Q,\mathcal{C})$ is said to be $\emph{multiplicity free}$ if every weight space has dimension at most one, i.e., $\dim_{\mathbb{K}}\SI(Q,\mathcal{C})_{\theta}\leq1$ for every weight $\theta\in\mathbb{Z}^{Q_0}$; otherwise it is said to be not multiplicity free.
        \item The algebra $A$ is said to be $\emph{multiplicity free}$ if for every dimension vector $\beta$ and every irreducible component $\mathcal{C}\subseteq \rep_{\beta}(Q,R)$, the semi-invariant ring $\SI(Q,\mathcal{C})$ is multiplicity free; otherwise it is said to be not multiplicity free (refer to \parencite[Definition 2.3.1]{ChinKinWey}).
    \end{enumerate}
\end{defn}
\begin{rmk}
    By (\ref{modsp}), the algebra $A$ being multiplicity free is equivalent to the fact that $\dim\mathcal{M}({\mathcal{C}})^{\theta-ss}=0$, for every irreducible component $\mathcal{C}$ and every weight $\theta$. 
\end{rmk}
\begin{thm}
    If there is an irreducible component $\mathcal{C}\subseteq \rep_{\beta}(Q,R)$ which is not an orbit closure and the coordinate ring $\mathbb{K}[\mathcal{C}]$ is a UFD, then the semi-invariant ring $\SI(Q,\mathcal{C})$ is not multiplicity free.
\end{thm}
\begin{proof}
    By Rosenlicht's result (refer to \parencite[Theorem]{Rosen}) we know that there exists a $\GL_{\beta}$-stable, open (dense) subset $U\subseteq\mathcal{C}$ such that there is a geometric quotient of $U$ by $\GL_{\beta}$, say $\Phi:U\rightarrow W$. By \parencite[Corollary, Section 3]{Fogarty} and \parencite[Remark 2, Section 2, Chapter 0]{mum}, we know that $W$ is an integral scheme of finite type over $\mathbb{K}$. Since $U$ is a union of infinitely many orbits, all of which are closed in $U$, we conclude that $\dim W\geq1$, and hence the transcendence degree of fraction field of $W$ is at least 1. By Rosenlicht's result (refer to \parencite[Theorem]{Rosen}), we know that the fraction field of $W$ can be identified with the $\GL_{\beta}$-invariant rational functions on $U$, i.e., $\mathbb{K}(U)^{\GL_{\beta}}$.
    
    By the argument in the previous paragraph we know that there exists a non-constant $\GL_{\beta}$-invariant rational function in $\mathbb{K}(U)^{\GL_{\beta}}$, say $p/q$. By Proposition \ref{ratSI}, we know that $p$ and $q$ are homogeneous semi-invariants of same weight, say $\theta$. Also, $p$ and $q$ are linearly independent, since the $\GL_{\beta}$-invariant rational function $p/q$ is a non-constant function. Hence, we conclude that $\dim_{\mathbb{K}}\SI(Q,\mathcal{C})_{\theta}\geq2$.
\end{proof}

\begin{rmk}
\phantomsection\label{lrmk}
Some immediate consequences of the above result are as follows:
    \begin{enumerate}
        \item It follows from the contrapositive statement of \parencite[Proposition 9]{ChinKinWey} that if the algebra $A$ is not multiplicity free, then there exists a dimension vector $\alpha$ such that $\rep_{\alpha}(Q,R)$ has infinitely many non-isomorphic bricks. In particular, it follows from \parencite[Theorem 4.2]{DIJ} that algebra $A$ is $\tau$-tilting infinite. 
        \item If we know that the algebra $A$ is $\tau$-tilting finite and there is an irreducible component $\mathcal{C}$ which is not an orbit closure, then the coordinate ring $\mathbb{K[\mathcal{C}]}$ is not a UFD. \label{notUFD}
    \end{enumerate}
\end{rmk}

\noindent
{\bf {Acknowledgement.}}
C.P., and D.W. were partially supported by the
National Sciences and Engineering Research Council of Canada and by the Canadian Defence
Academy Research Programme.

\newpage
\printbibliography

@article{King,
    author = {King, A. D.},
    title = "{Moduli Of Representations Of Finite Dimensional Algebras}",
    journal = {The Quarterly Journal of Mathematics},
    volume = {45},
    number = {4},
    pages = {515-530},
    year = {1994},
    month = {12},
    issn = {0033-5606},
    
}

@article{ChinKinWey,
    author = {Chindris, Calin and Kinser, Ryan and Weyman, Jerzy},
    title = "{Module Varieties and Representation Type of Finite-Dimensional Algebras}",
    journal = {International Mathematics Research Notices},
    volume = {2015},
    number = {3},
    pages = {631-650},
    year = {2013},
    month = {10},
    issn = {1073-7928},
    
}

@article{satkim,
author = {T. Kimura and M. Sato},
title = {{A classification of irreducible prehomogeneous vector spaces and their relative invariants}},
volume = {65},
journal = {Nagoya Mathematical Journal},
publisher = {Nagoya Mathematical Journal},
pages = {1 -- 155},
year = {1977},

}

@book{weyderk,
  title={An introduction to quiver representations},
  author={Derksen, Harm and Weyman, Jerzy},
  volume={184},
  year={2017},
  publisher={American Mathematical Society}
}

@Article{weypaq,
author={Paquette, Charles
and Weyman, Jerzy},
title={Isotropic Schur Roots},
journal={Transformation Groups},
year={2018},
month={09},
day={01},
volume={23},
number={3},
pages={841-874},
issn={1531-586X},

}

@book{SanRitt,
author = {Santos, Walter and Rittatore, Alvaro},
year = {2010},
month = {07},
pages = {},
title = {Actions and Invariants of Algebraic Groups},
isbn = {9780429135736},

}

@article{Nagata,
author = {Nagata, Masayoshi},
title = {{A remark on the unique factorization theorem.}},
volume = {9},
journal = {Journal of the Mathematical Society of Japan},
number = {1},
publisher = {Mathematical Society of Japan},
pages = {143 -- 145},
year = {1957},

}

@article{Rosen,
  title={A remark on quotient spaces},
  volume = {35},
  pages = {486 -- 489},
  author={Rosenlicht, Maxwell},
  journal={An. Acad. Brasil Cien{\^c}},
  year={1963}
}

@article{Fogarty,
title = {Geometric quotients are algebraic schemes},
journal = {Advances in Mathematics},
volume = {48},
number = {2},
pages = {166-171},
year = {1983},
issn = {0001-8708},
author = {John Fogarty}
}

@book{mum,
  title={Geometric Invariant Theory},
  author={Mumford, D. and Fogarty, J. and Kirwan, F.},
  isbn={9783540569633},
  year={1994},
  publisher={Springer Berlin Heidelberg}
}

@article{KAC,
title = {Infinite root systems, representations of graphs and invariant theory, II},
journal = {Journal of Algebra},
volume = {78},
number = {1},
pages = {141-162},
year = {1982},
issn = {0021-8693},
author = {V.G. Kac}
}

@article{DIJ,
    author = {Demonet, Laurent and Iyama, Osamu and Jasso, Gustavo},
    title = "{{$\tau$}-Tilting Finite Algebras, Bricks, and {$g$}-Vectors}",
    journal = {International Mathematics Research Notices},
    volume = {2019},
    number = {3},
    pages = {852--892},
    year = {2017},
    month = {07},
    issn = {1073-7928},
}
\end{document}